\pdfoutput=1
\documentclass[11pt]{amsart}
\usepackage{amsmath,amsfonts,amssymb,amsthm}
\usepackage[margin=1in]{geometry}
\usepackage{epsfig}
\usepackage{graphics}
\usepackage{subfig} 
\usepackage{comment}
\usepackage{tikz}

\newtheoremstyle{thm}
  {9pt}{9pt}{\itshape}{}{\bfseries}{}{.5em}{}

\theoremstyle{thm}
\newtheorem{thm}{Theorem}[section]
\newtheorem{cor}[thm]{Corollary}
\newtheorem{lem}[thm]{Lemma}

\newtheorem{lemma}[thm]{Lemma}

\newtheoremstyle{defin}
  {9pt}{9pt}{}{}{\bfseries}{}{.5em}{}
\theoremstyle{defin}

\newtheorem{cl}{Claim}
\newtheoremstyle{exm}
  {9pt}{9pt}{}{}{\scshape}{}{.5em}{}
\theoremstyle{exm}

\newtheoremstyle{proof}
  {}{}{}{}{\itshape}{:}{.5em}{}
\theoremstyle{proof}

\def\co{\mathcal O}

\def\cR{\mathcal R}
\def\cT{\mathcal T}

\def\cT{\mathbf{T}}

\def\<{\langle}
\def\>{\rangle}

\def\y{ {\text {\rm y}  } }

\def\0{{\mathbf 0}}

\def\.{\hskip.06cm}
\def\ts{\hskip.03cm}

\def\T{{{\mathbf{T}}}}
\def\R{{{\mathbb{R}}}}
\def\W{{{\mathbf{W}}}}

\def\co-NP{\textup{co-NP}}
\def\NP{\textup{NP}}
\def\P{\textup{P}}
\def\SP{\textup{\#P}}
\def\G{\Gamma}

\def\cT{\mathcal{T}}
\def\cR{\mathcal{R}}
\def\Z{\mathbb{Z}}
\def\N{\mathbb{N}}

\def\p{\pi}

\newcommand{\abs}[1]{\left\lvert#1\right\rvert}
\newcommand{\bs}{\backslash}

\newcommand{\latin}[1]{\textsl{#1}}

\newcommand{\fig}[1]{Figure~\ref{fig:#1}}
\newcommand{\problem}[1]{\textsc{#1}}

\newcommand{\problemdef}[3]{
\medskip
\begin{tabular}{ll}
\multicolumn{2}{l}{\problem{#1}}\\
\textbf{Instance:} & #2 \\
\textbf{Decide:} & #3
\end{tabular}\medskip}

\newcommand{\tf}[1]{\mathsf{#1}} 
\newcommand{\tsf}[1]{\mathbf{#1}} 
\newcommand{\tmf}[1]{\mathcal{#1}} 
\newcommand{\B}{\tsf B}
\newcommand{\J}{\tf J}
\newcommand{\K}{\tf K}
\renewcommand{\L}{\tf L}
\renewcommand{\H}{\tf H}
\renewcommand{\R}{\tf R}
\renewcommand{\P}{\tf P}
\newcommand{\M}{\tmf M}
\newcommand{\MM}{\tsf M_\tmf M}
\newcommand{\x}{\mathbf x}

\newcommand{\tilen}[7]{
\begin{scope}[xshift=#5cm,yshift=#6cm]
\draw (0,0) rectangle (1,1);
\node[left] at (0,0.5) {$#1$};
\node[above] at (0.5,1) {$#2$};
\node[below] at (0.5,0) {$#3$};
\node[right] at (1,0.5) {$#4$};
\node at (0.5,0.5) {$#7$};
\end{scope}
}
\newcommand{\tiletn}[7]{
\tilen{#1}{#2}{#3}{#4}{#5}{#6}{\tf{#7}}
}
\newcommand{\tile}[6]{
\tilen{#1}{#2}{#3}{#4}{#5}{#6}{}
}

\title{Rectangular tileability and complementary tileability are undecidable}
\author[Jed~Yang]{ \ Jed~Yang$^\star$}
\thanks{\thinspace ${\hspace{-.45ex}}^\star$Department of Mathematics, UCLA, Los Angeles, CA 90095, USA; \.
\texttt{jedyang@ucla.edu}}
\date{}

\begin{document}

\begin{abstract}
Does a given a set of polyominoes tile some rectangle?
We show that this problem is undecidable.
In a different direction, we also consider tiling a cofinite subset of the plane.
The tileability is undecidable for many variants of this problem.
However, we present an algorithm for testing whether the complement of a finite region is tileable by a set of rectangles.
\end{abstract}
\maketitle

\section{Introduction}

Tileability on the plane has been a subject of much study~\cite{GS}.
A lot of focus was in tilings on the square grid by polyominoes~\cite{Gol-book},
including 
establishing \NP-completeness~\cite{Lew,MR,PY-rect}
and
finding efficient algorithms when possible~\cite{KK,Rem}.
Aperiodicity in tilings of the entire plane has also been well-studied~\cite{Moz,Pen},
with connections to ergodic theory~\cite{Rad,CR} and quasicrystals~\cite{DS}.
Recently, a single (disconnected) tile that exhibits aperiodic behavior was found~\cite{ST},
partially settling a famous open problem.

Can the plane be tiled using translated copies of a given set of polyomino tiles?
Berger showed that this decision problem is undecidable~\cite{Ber},
meaning that there is no general algorithm that can always answer this question from the input.
This implies that there exists \emph{aperiodic tilesets},
\latin{i.e.}, tiles that can \emph{only} tile the plane without translational symmetry.
Indeed, Berger provided an aperiodic tileset of $20426$ tiles,
and Robinson reduces this number to $6$ if rotations and reflections are allowed in addition to translations~\cite{Rob}.
This disproves a conjecture of Wang (see Subsection~\ref{ss:wang}).

To have aperiodicity and undecidability, clearly the complexity of tiles and tilings must increase without bound.
The following result shows that one can encode this complexity in a single tile alone.%
\footnote{We require that this tile be used precisely once.
To that end, we consider it as an input and tile its complement by a fixed tileset.}

\begin{thm}[Complementary Tileability]
There is a tileset $\T$ such that it is undecidable whether
the complement of a finite input region $\G$ is tileable by~$\T$.
\end{thm}

Instead of tiling the entire plane or a cofinite subset, we also consider tiling finite regions.
If a rectangle is tileable, then of course the plane is tileable.
Thus the following is a variation of the result above.

\begin{thm}[Rectangular Tileability]\label{t:rect}
It is undecidable whether a given tileset $\T$ can tile \emph{some} rectangle.
\end{thm}

This should be contrasted with tiling a given rectangle by a fixed tileset:

\begin{thm}[\cite{LMP}]\label{t:curious}
Tileability of an $[n\times m]$ rectangle by a fixed tileset can be determined in time $O(\log n+\log m)$.
\end{thm}

We discuss this connection and some curious consequences of Theorem~\ref{t:rect} in Subsection~\ref{ss:curious}.
It is worth noting that if we are given a finite region to tile as opposed to the entire plane,
the problem is decidable simply with exhaustive search (see Subsection~\ref{ss:complexity} for results in this finite setting).
However, although the region to be tiled in \problem{Rectangular Tileability} is finite,
the problem is (potentially) undecidable as the finite region is unspecified.
That is, we are tiling a finite object from an infinite collection.
Indeed, Theorem~\ref{t:rect} shows that \problem{Rectangular Tileability} is undecidable.
We prove this in Section~\ref{s:rect},
where we moreover show that the problem remains undecidable when the size $\abs{\T}$ of the tileset is fixed.
Also, as a corollary,
we see that tileability of a unit square by finitely many similar copies of tiles
(where rotations, reflections, and dilations are allowed in addition to translations) is also undecidable.

In Section~\ref{s:aug}, we prove undecidability for \problem{Augmentability},
where we are to \emph{augment} a finite simply connected region~$\G$ by tiles so that the union is tileable.
Any augmentable region $\G$ by the (horizontal and vertical) dominoes must necessarily be \emph{balanced},
\latin{i.e.}, has the same number of black and white squares when the plane is colored as a checkerboard.
Korn showed that this is not sufficient, and also proved that $\G$ is augmentable if it is \emph{row-convex},
\latin{i.e.}, each horizontal row forms a single contiguous region~\cite[\S11]{Korn}.
This should be compared with Theorem~\ref{t:alg} below.

Usually once a result is established for a decision problem with several inputs,
one may consider fixing some of these inputs and aim to obtain the same conclusion.
To that end, we fix the tileset and instead let the region vary as the input of \problem{Tileability}.
However, decidability makes sense only if the input is finite, yet the region is infinite.
As such, in Section~\ref{s:cofinite},
we consider some variations of tiling cofinite regions.
Though most of these problems are undecidable,
in the positive direction, we provide an algorithm for \problem{Tileability} of cofinite regions by arbitrary sets of rectangular tiles.
\begin{thm} \label{t:alg}
It is decidable whether the complement of a given finite region $\G$ is tileable by a given tileset~$\T$ consisting only of rectangles.
\end{thm}
In contrast, we show in Section~\ref{s:quadrant} that \problem{Tileability} of \emph{indented quadrants} by rectangles is undecidable.

\section{Basic definitions}\label{s:def}
We call a subset of $\Z^2$ a \emph{region}.
By identifying $\Z^2$ as a union of closed unit squares in $\mathbb{R}^2$ centered at the integer lattice points,
a region takes on a geometric \emph{shape} in the obvious manner.
We freely switch between viewpoints when convenient.%
\footnote{For example, \emph{finite} and \emph{disjoint} refer to regions as subsets of~$\Z^2$,
so the shapes of disjoint regions (\latin{e.g.}, tiles in a tiling) may intersect on their boundaries,
but \emph{simply connected} refers to the shapes of regions.}
We say a finite region is an (\emph{ordinary}) \emph{tile} if its shape is simply connected (s.c.).
A finite collection of tiles is called a \emph{tileset}.
Note that $\Z^2$ acts naturally on a tile by translation.
Given a region $\G$ and a tileset $\T$, a \emph{$\T$-tiling} is a partition of $\G$ into translated copies of tiles in~$\T$.
A region is \emph{$\T$-tileable} if it admits a $\T$-tiling.
When the tileset $\T$ is understood, we may suppress the prefix for notational convenience.

The \emph{boundary} of (the shape of) a tile consists of unit-length horizontal and vertical \emph{edges}.
A (\emph{generalized}) \emph{Wang tile} is a tile whose edges are labeled with symbols that are referred to as \emph{colors}.
When tiling with Wang tiles, incident edges must be the same color.
When a Wang tile is a single square, it is also called a \emph{Wang square}.%
\footnote{In the literature, a \emph{square Wang tile} is usually simply called a \emph{Wang tile} and sometimes called a \emph{domino} despite being a single square.}
To emphasize the distinction between ordinary tiles and Wang tiles, we may refer to the former as \emph{polyominoes}.

Consider the following decision problem:

\problemdef{Tileability}
{A tileset $\T$ and a region $\G$.}
{Does $\G$ admit a $\T$-tiling?}

Berger proved that \problem{Tileability} for Wang squares is undecidable when $\G=\Z^2$ is the whole plane~\cite{Ber}.
By straight-forward reductions (see Lemma~\ref{l:reduce}),
most problems involving tileability of polyominoes, generalized Wang tiles, and Wang squares are equivalent.
In this paper we will consider a few other tileability problems that are undecidable.
We will usually state theorems for polyominoes, but prove them (first) for Wang tiles,
then appeal to the reductions in the appendix.

\section{Rectangular Tileability}\label{s:rect}

Here we consider the decision problem \problem{Rectangular Tileability},
where the input is a tileset and the region to be tiled is unspecified.

\problemdef{Rectangular Tileability}
{A tileset~$\T$.}
{Does there exist \emph{some} rectangle that is tileable by $\T$?}

By an \emph{alphabet} $\Sigma$ we mean a set of symbols,
whose elements could be juxtaposed to form \emph{words}.
Let $\Sigma^*$ denote the set of all finite words in the alphabet~$\Sigma$.
The \problem{Post Correspondence Problem} is a well-known undecidable decision problem:

\problemdef{PCP}
{An alphabet $\Sigma$, positive integer $n$, and words $j_t,k_t\in\Sigma^*$ for $t\in[n]$.}
{Does there exists a positive integer $d$ and $e_i\in[n]$ for $i\in[d]$, such that\\
&the concatenated words $j_{e_1}j_{e_2}\ldots j_{e_d}$ and $k_{e_1}k_{e_2}\ldots k_{e_d}$ are the same?}

\begin{thm}[\cite{Post}]
\problem{PCP} is undecidable.
\end{thm}

Given a \problem{PCP} instance~$I$,
we will construct a tileset~$\T$
such that $\T$ tiles some rectangle if and only if $I$ admits a solution,
and thus proving Theorem~\ref{t:rect}.

\begin{lemma}\label{l:wang}
Given a \problem{PCP} instance~$I$,
there is a computable%
\footnote{The main proof technique of undecidability is by reducing a known undecidable problem to the current problem.
This transformation must be done in a computable manner, \latin{i.e.}, by a deterministic algorithm.
Thus \emph{computable} means that the existence of the object in question can be substantiated by an explicit construction.}
tileset $\W$ of generalized Wang tiles
such that $\W$ tiles a white rectangle if and only if $I$ admits a solution.
\end{lemma}
\begin{proof}
For each $j_t=x_1x_2\ldots x_r\in\Sigma^*$,
create a $1\times r$ tile $\J_t$ whose top boundary edges are all colored~$U$, the bottom color sequence is $(x_1,t),x_2,\ldots,x_r$,
and the vertical sides are colored~$V$.
Similarly for $k_t=y_1y_2\ldots y_s\in\Sigma^*$,
create a $1\times s$ tile $\K_t$ whose top boundary is $(y_1,t),y_2,\ldots,y_s$, the bottom colored~$D$,
and the vertical sides are colored~$V$ (\fig{pcp}).
For each $x\in\Sigma$ and $t\in[n]$, create the following six \emph{transmitter tiles} in \fig{tx}.
Finally, add in the \emph{border tiles} in \fig{border}, which are the only ones with \emph{white} (unlabeled) colors.
This finishes the construction of the tileset~$\W$.
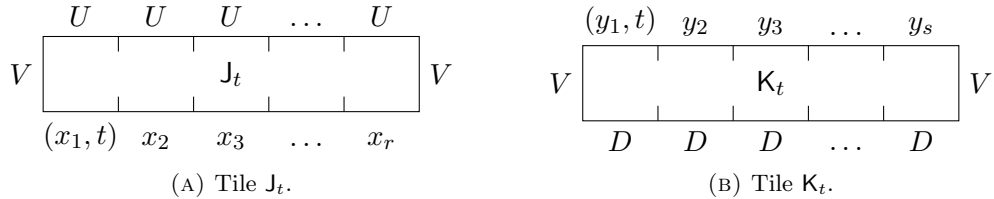
\begin{figure}[hbtp]
\subfloat[Tile $\J_t$.]{
   \begin{tikzpicture}
   \draw (0,0) rectangle (5,1);
   \foreach \x in {1,2,...,4}
   {
      \draw (\x cm, 0 cm) -- (\x cm, 0.2 cm);
      \draw (\x cm, 1 cm) -- (\x cm, 0.8 cm);
   }
   \node at (2.5,0.5) {$\J_t$};

   \node[below] at (0.5,0) {$(x_1,t)$};
   \foreach \x in {1,2,...,3}
      \node[above,xshift=\x cm] at (-0.5,1) {$U$};
   \foreach \x in {2,3,...,3}
      \node[below=4pt,xshift=\x cm] at (-0.5,0) {$x_\x$};
   \node[below=8pt] at (3.5,0) {$\ldots$};
   \node[above] at (3.5,1) {$\ldots$};
   \node[below=4pt] at (4.5,0) {$x_r$};
   \node[above] at (4.5,1) {$U$};

   \node[left] at (0,0.5) {$V$};
   \node[right] at (5,0.5) {$V$};
   
   \end{tikzpicture}
}
\qquad
\subfloat[Tile $\K_t$.]{
   \begin{tikzpicture}
   \draw (0,0) rectangle (5,1);
   \foreach \x in {1,2,...,4}
   {
      \draw (\x cm, 0 cm) -- (\x cm, 0.2 cm);
      \draw (\x cm, 1 cm) -- (\x cm, 0.8 cm);
   }
   \node at (2.5,0.5) {$\K_t$};

   \node[above] at (0.5,1) {$(y_1,t)$};
   \foreach \x in {1,2,...,3}
      \node[below,xshift=\x cm] at (-0.5,0) {$D$};
   \foreach \x in {2,3,...,3}
      \node[above=0pt,xshift=\x cm] at (-0.5,1) {$y_\x$};
   \node[below=6pt] at (3.5,0) {$\ldots$};
   \node[above=0pt] at (3.5,1) {$\ldots$};
   \node[above=0pt] at (4.5,1) {$y_s$};
   \node[below] at (4.5,0) {$D$};

   \node[left] at (0,0.5) {$V$};
   \node[right] at (5,0.5) {$V$};
   
   \end{tikzpicture}
}
   \caption{PCP tiles.}
   \label{fig:pcp}
\end{figure}

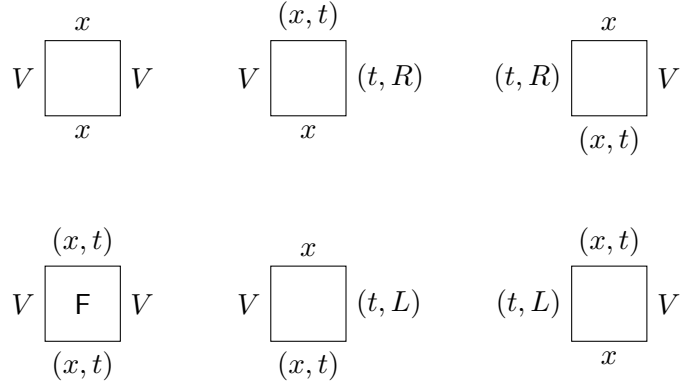
\begin{figure}[hbtp]
   \begin{tikzpicture}
   \tile{V}{x}{x}{V}03
   \tile{V}{(x,t)}{x}{(t,R)}33
   \tile{(t,R)}{x}{(x,t)}{V}73
   \tilen{V}{(x,t)}{(x,t)}{V}00{\tf F}
   \tile{V}{x}{(x,t)}{(t,L)}30
   \tile{(t,L)}{(x,t)}{x}{V}70
   \end{tikzpicture}
   \caption{Transmitter tiles.}
   \label{fig:tx}
\end{figure}

\begin{figure}[hbtp]
   \begin{tikzpicture}
   \tile{}{}{L}{T}06
   \tile{T}{}{U}{T}36
   \tile{T}{}{R}{}66
   \tile{}{L}{L}{V}03
   \tile{V}{R}{R}{}63
   \tile{}{L}{}{B}00
   \tile{B}{D}{}{B}30
   \tile{B}{R}{}{}60
   \end{tikzpicture}
   \caption{Border tiles.}
   \label{fig:border}
\end{figure}
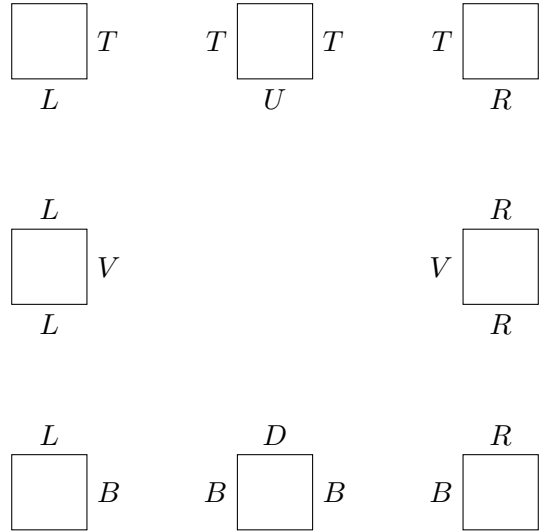

If $I$ has a solution $e_1,\ldots,e_d$,
we first put $\J_{e_1},\J_{e_2},\ldots,\J_{e_d}$ in a row on top,
$\K_{e_1},\K_{e_2},\ldots,\K_{e_d}$ in a row on the bottom.
If we view the color $(x,t)$ as $x$, then the bottom of the top row and the top of the bottom row have the same color sequence.
The transmitter tiles allow ``shifting'' the tag $t$ to the left or right one step at a time until they match up as well.
We thus have a rectangle with the top and bottom colored $U$ and $D$, respectively, while the vertical sides are colored~$V$.
Notice that the border tiles can tile a rectangular border of any size, whose outside boundary is white,
and the inside colors match exactly what we have.

Conversely, suppose $\W$ tiles a white rectangle.
Consider the smallest white rectangle it can tile.
The border tiles occur only on the boundary.
Indeed, if there is any border tile in the interior, by construction, it must form a white rectangular frame with filled interior,
a contradiction to the minimality.
Now remove the border tiles on the boundary,
the top (resp.\ bottom) row must be a sequence of $\J_t$ (resp.\ $\K_t$) tiles.
Let $e_1,\ldots,e_d$ be the indices of the top sequence.
By construction of the transmitter tiles, the bottom sequence share the same indices.
Furthermore, viewing the color $(x,t)$ as $x$, the color sequences are the same.
This means that the concatenated words coincide, and we indeed extracted a solution to~$I$.
\end{proof}

\begin{cor} \label{cor:poly}
Given a \problem{PCP} instance~$I$,
there is a computable tileset $\T$ of polyominoes
such that $\T$ tiles a rectangle if and only if $I$ admits a solution.
\end{cor}
\begin{proof}
Take the set $\W$ of generalized Wang tiles afforded by Lemma~\ref{l:wang},
and transform it to a set $\T$ of polyomino tiles using Lemma~\ref{l:reduce} in the appendix with a minor change.
In the construction, we will replace the color white by a straight boundary,
including the omission of the corner zig-zags.

If $I$ admits a solution, then there is a tiling of a white rectangle by~$\W$.
Replace each Wang tile by its corresponding polyomino in~$\T$.
By construction, they fit together to form a tiling of a rectangle by~$\T$.

Conversely, suppose $\T$ tiles some rectangle.
Consider the smallest rectangle tileable by $\T$ and fix one such tiling.
Note that the only tiles that can touch the boundary of the rectangle are the border tiles
(\latin{i.e.}, those tiles in~$\T$ that correspond to the border tiles in~$\W$),
lest there be small holes that cannot be filled.
Conversely, the border tiles cannot occur in the interior by minimality,
since they always form rectangular frames.
We have replicated the exact same situation as in the proof of Lemma~\ref{l:wang},
allowing us to extract a minimal Wang tiling of a white rectangle,
and thus a solution to~$I$, as desired.
\end{proof}

This concludes the proof of Theorem~\ref{t:rect}.
Moreover, we can allow the pieces to be rotated, reflected, or even dilated.

\begin{cor} \label{cor:similar}
It is undecidable whether a tileset $\T$ can tile a unit square with finitely%
\footnote{If infinitely many similar copies are allowed, any tileset can tile a unit square greedily in a manner akin to that of an Apollonian gasket.}
many similar copies of its tiles.
\end{cor}
\begin{proof}
We claim that the specific tiles $\T$ in the proof of Corollary~\ref{cor:poly} tile a rectangle by translation alone if and only if 
it admits a similar tiling of the unit square.

Suppose $\T$ tiles some $[a\times b]$ rectangle by translation.
By repeating the rectangle $ab$ times, we have a tiling of an $[ab\times ab]$ square.
Dilating, we obtain a similar tiling of the unit square.

Conversely, suppose $\T$ admits a similar tiling of the unit square.
Let $S$ be the union of all tiles with the minimum scaling factor.
All the corner zig-zag of tiles in $S$ must be matched by other tiles of the same scaling factor.
As such, all tiles that touch the boundary of $S$ must be border tiles.
In the proof of Lemma~\ref{l:reduce},
we note that the interlocking corner zig-zags force all tiles to have the same orientation.
This works even when we removed the corner zig-zags for the white edges.
In this case, each border tile must be part of a rectangular frame with all its border tiles in the same orientation.
Since there are finitely many rectangular frames, there exists one without border tiles within.
This rectangular frame is fully tiled in its interior by tiles at this minimum scaling factor.
Indeed, if there were any gaps, that would be a boundary of $S$, a contradiction.
Finally, note that all tiles in this rectangular region has the same orientation.
Thus $\T$ tile some rectangle by translation alone, as desired.
\end{proof}

Theorem~\ref{t:rect} remains true even if the number of tiles is bounded.

\begin{thm}
\problem{Rectangular Tileability} with $19$ tiles is undecidable.%
\footnote{Of course, the number of tiles can be incremented by splitting a tile in such a way that the pieces must reassemble to form the previous tile, or by adding tiles that \emph{cannot} be used in any tiling of rectangles.}
\end{thm}
\begin{proof}
We first prove that some finite number suffices.
This comes directly from the fact that \problem{PCP} remains undecidable when the number of pairs $n$ is fixed, for $n\geq7$~\cite{MS}.
Of course, by using a binary encoding of the alphabet, we may assume $m=\abs{\Sigma}=2$.
By making sure that this encoding is of even length, we may omit the transmitter tile $\tf F$ that fixes $(x,t)$, and let the labels oscillate back and forth.
This means we need only $2n+m+4nm+8$ tiles.
Taking $n=7$ and $m=2$, we see that $80$ tiles suffices.

Now we briefly sketch a proof of the claimed $19$ bound.
In~\cite{Oll}, Ollinger constructs a tileset $\T$ of $11$ polyominoes for an arbitrary Wang tileset~$\W$,
such that $\T$ tiles the plane if and only if $\W$ tiles the plane.
By adding $8$ (different) border tiles similar to the idea used above,
we obtain a tileset $\T'$ of $19$ polyominoes such that
$\T'$ tile some rectangle if and only if $\W$ tile some white rectangle.
It is worth noting that in~\cite{Oll}, the Wang squares are arranged on the square lattice grid as if they are (Aztec) diamonds.
Thus we must create a new set of ``diamond'' Wang tiles to fit this setup,
and create the border tiles to create a periodic border that fits Aztec diamonds.
We omit the easy details.
\end{proof}

\section{Complementary Tileability}\label{s:cofinite}

\subsection{Tiling cofinite regions with a fixed tileset}
Now we will consider a series of \emph{complementary} tiling problems.
First, as motivated in Section~\ref{s:def}, we consider tiling a cofinite region,
given by its finite complement in the plane.

\problemdef{Complementary Tileability}
{A tileset~$\T$ and a finite region $\G$.}
{Does the cofinite region $\Z^2\bs\G$ admit a $\T$-tiling?}

We will see that this is undecidable, by showing that the problem is undecidable even if the tileset~$\T$ is fixed:

\newcommand{\ctt}[1]{\problem{Complementary $#1$-Tileability}}
\problemdef{\ctt{\T}}
{A finite region $\G$.}
{Does the cofinite region $\Z^2\bs\G$ admit a $\T$-tiling?}

\begin{thm}\label{t:ctt}
There exists a tileset~$\T$ such that \ctt{\T} is undecidable.
\end{thm}
\begin{proof}
Fix a universal Turing machine~$\M$ on an alphabet~$\Sigma$ with blank symbol $0\in\Sigma$.
Treating the initial tape configuration as the input,
it is undecidable whether $\M$ halts.
Let $\x=x_1x_2\ldots x_r\in\Sigma^*$ be a word in~$\Sigma$.
Construct region $\G_\x$ as shown in \fig{start}, where $s$ is the starting state of~$\M$.
Using the emulation of Turing machines by Wang tiles in Lemma~\ref{l:tm},
we obtain a Wang tileset~$\MM$ that emulates~$\M$.
Add the additional tiles in \fig{add}, and pass to polyominoes by Lemma~\ref{l:reduce} along with input region~$\G_\x$.
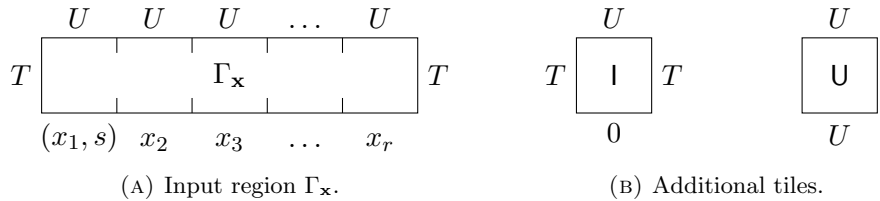
\begin{figure}[hbtp]
\subfloat[Input region $\G_\x$.]{
   \begin{tikzpicture}
   \draw (0,0) rectangle (5,1);
   \foreach \x in {1,2,...,4}
   {
      \draw (\x cm, 0 cm) -- (\x cm, 0.2 cm);
      \draw (\x cm, 1 cm) -- (\x cm, 0.8 cm);
   }
   \node at (2.5,0.5) {$\G_\x$};

   \node[below] at (0.5,0) {$(x_1,s)$};
   \foreach \x in {2,3,...,3}
      \node[below=4pt,xshift=\x cm] at (-0.5,0) {$x_\x$};
   \node[below=4pt] at (4.5,0) {$x_r$};

   \foreach \x in {1,2,...,3}
      \node[above,xshift=\x cm] at (-0.5,1) {$U$};
   \node[above] at (4.5,1) {$U$};

   \node[below=8pt] at (3.5,0) {$\ldots$};
   \node[above] at (3.5,1) {$\ldots$};

   \node[left] at (0,0.5) {$T$};
   \node[right] at (5,0.5) {$T$};
   \end{tikzpicture}
   \label{fig:start}
}
\qquad
\subfloat[Additional tiles.]{
   \begin{tikzpicture}
   \tiletn{T}{U}{0}{T}00{I}
   \tiletn{}{U}{\vphantom{()}U}{}30{U}
   \end{tikzpicture}
   \label{fig:add}
}
\caption{Initial tiles.}
\end{figure}

In a tiling, the left and right of the input region can only be tiled by (the polyomino associated with)~$\tf I$,
thus initializing the tape with the blank symbol $0$ bi-infinitely,
and separating the plane into two halves.
The upper half will be tiled by~$\tf U$.
The lower half is tileable if and only if there is a non-halting computation of~$\M$.
\end{proof}

\subsection{Tiling fixed cofinite regions}
Obviously, for some tileset~$\T$, such as a single square, this problem is decidable.
In contrast, we can fix the region $\G$ instead of the tileset~$\T$:

\newcommand{\gct}[1]{\problem{$#1$-Complementary Tileability}}
\problemdef{\gct{\G}}
{A tileset~$\T$.}
{Does the region $\Z^2\bs\G$ admit a $\T$-tiling?}

\begin{thm}
\gct{\G} is undecidable for any finite simply connected%
\footnote{Simple connectivity is necessary.
Indeed, if $\G$ has a missing unit square in its interior,
then its complement is tileable if and only if the tileset contains a single unit square.}
region~$\G$.
\end{thm}

In particular, even if we fix $\G$ as a single square, denoted $[1\times1]$, this problem is still undecidable.%
\footnote{This is not surprising, since tiling the entire plane ($\G=\varnothing$) is undecidable~\cite{Ber}.
In fact, the main difficulty was the lack of a starting point.
The case we consider is essentially the \emph{origin-constrained} version of tileability,
which is undecidable~\cite{Wang-aea}.
We include a sketch for completeness.}

\begin{proof}[Sketch of proof]
We modify \problem{Complementary Tileability} to have a fixed starting region.
Indeed, suppose we are given an instance consisting of a tileset $\T$ and a region~$\G$.
Let $\T'$ be obtained by passing $\T\cup\{\G\}$ to Wang tiles and back to polyominoes via Lemma~\ref{l:reduce}.
Scale all these tiles by a factor of~$2$,
and then remove a single square from the tile $\G'$ corresponding to~$\G$.
Note that $\Z^2\bs[1\times1]$ admits a $\T'$-tiling if and only if $\Z^2\bs\G$ admits a $\T$-tiling.
Indeed, by passing to Wang tiles and back, these tiles automatically align in a dilated integer lattice grid in any tiling.
Since all tiles were dilated by a factor of~$2$,
if $\G'$ is not used, there is no way to tile around the $[1\times1]$ unit square.
On the other hand, every occurrence of $\G'$ will leave a $[1\times1]$ square uncovered.
Thus $\G'$ will be used precisely once, matching the square at where it was removed.
This transforms the original complementary tiling problem to one with fixed input region~$[1\times1]$.

Fixing $\G$ as something else is similar.
Indeed, let $S$ be a large $[N\times N]$ square containing $\G$, such that $S'=S\bs\G$ is connected.
It is possible that $S'$ is not simply connected, in which case we can break $S'$ into two s.c.\ pieces that will interlock with each other only.
Given an instance $\T$ of \gct{[1\times1]}, scale all tiles in $\T$ by a factor of $N$ to obtain an equivalent \gct{[N\times N]} problem.
Now add (the two pieces of) $S'$ to the scaled tileset, reducing to \gct{\G}, as desired.
\end{proof}

\subsection{Tiling cofinite regions with rectangles}
On the positive side, somewhat surprisingly, when the tileset consists only of rectangles,
the most general case (with neither the region nor the tileset fixed) is actually decidable.

\problemdef{Complementary Tileability with Rectangles}
{A region $\G$ and a tileset~$R$ of rectangles.}
{Does the region $\Z^2\bs\G$ admit a $R$-tiling?}

\begin{proof}[Proof of Theorem~\ref{t:alg}]
We first describe the algorithm.
Suppose we are given a region $\G$ and a tileset $\T$ consisting only of rectangles.
Let $A$ be the smallest rectangular region that contains $\G$ in its interior.
Place tiles to cover $A$ (where tiles are allowed to protrude outside of $A$), such that $\G$ is uncovered.
Formally, find a region $B\supset A$ and a tiling of~$B\bs\G$.
Even though this may look like an infinite problem with $B$ unspecified,
note that it is really a finite problem, since we can exhaustively search all local tilings around~$\G$.
Either there is no way to cover $A$ while avoiding $\G$, in which case there is no tiling of~$\Z^2\bs\G$,
or we find such a $B$ and a tiling $\p$ of $B\bs\G$, in which case the complement of $\G$ is tileable.

Indeed, we may assume that $\p$ is minimal in the sense that removing any tile will leave parts of $A$ uncovered.
Consider the tiling $\p$ at the top boundary of~$A$.
It consists of some rectangles sticking out, and some that tile $A$ just right.
Regardless, simply use each of these rectangles to tile the infinite half-strip above each segment.
Similarly, tile the three half-strips from the other three edges.
Now we are left with four quadrants, which are obviously tileable.
\end{proof}

\section{Tiling indented quadrants with rectangles}\label{s:quadrant}
Lemma~\ref{l:reduce} lists several classes of \problem{Tileability} problems that are equivalent.
Tiling with rectangles is equivalent with the others for finite regions~\cite{PY-rect},
but Theorem~\ref{t:alg} puts an end to the hope of extending the equivalence to cofinite regions.
Here we present an equivalence for special infinite regions, and use it to exhibit yet another undecidable result.
We combine an explicit construction in~\cite{PY-rect} with ideas already appearing in this paper.
Thus to save space and avoid redundancy, we only sketch the proofs; details will appear in~\cite{Yang}.

\begin{lem}
Given a (Wang) tileset $\T$, it is undecidable whether the (white) fourth quadrant is $\T$-tileable.
\end{lem}
\begin{proof}
Take a Turing machine $\M$, and use Lemma~\ref{l:tm} to get an associated tileset~$\MM$.
Similar to the proof of Theorem~\ref{t:ctt},
one can easily make some modifications to obtain a tileset $\MM'$
with the property that the fourth quadrant is $\MM'$-tileable
if and only if the Turing machine has a non-halting computation when started on a 1-way infinite blank tape.
Since that is undecidable (with $\M$ varying as input), tileability of the quadrant is undecidable.
\end{proof}

As usual, this result holds for a tileset of polyominoes.
Obviously, it makes no sense to ask the same question for rectangular tiles.
However, if we replace the boundary of the quadrant by periodic rectilinear curves,
then the problem becomes undecidable again.
Call such a boundary a \emph{periodically indented} fourth quadrant.

\begin{thm}\label{t:quadrant}
Given a periodically indented fourth quadrant $\G$ and a tileset $\T$ of rectangles,
it is undecidable whether $\G$ admits a $\T$-tiling.
\end{thm}

\begin{proof}
Take a set of polyominoes~$\T$.
The reduction in~\cite{PY-rect} provides a set of rectangles~$\T'$, and a linear-time function $f$ that transforms a finite region $\G$ to another finite region~$\G'$,
such that $\G$ is $\T$-tileable if and only if $\G'$ is $\T'$-tileable.
Explicitly, $f$ takes each unit-length edge of $\G$ and replaces it with a rectilinear curve to get~$\G'$.
By feeding a single vertical and a single horizontal edge to~$f$,
we get two rectilinear curves that can be used to make the periodic boundary of an indented fourth quadrant.
By \emph{carefully} following the proof of the bijective correspondence of $\T$-tilings of regions and $\T'$-tilings of the transformed regions,
one can see that everything works for this specific infinite case as well.
\end{proof}

We should note that the quadrant having a corner is essential when applying the \emph{proof} in~\cite{PY-rect} to this context.
For example, this reduction does not work for the half-plane.
Indeed, for that case, the sketch of the lemma above does not work, but can be fixed with some additional ideas to prove the corresponding lemma.
The theorem, however, cannot be salvaged, as the algorithm of Theorem~\ref{t:alg} would apply with some minor alterations.

\section{Augmentability}\label{s:aug}

In the style of \problem{Rectangular Tileability},
where we are interested if there is \emph{some} finite (rectangular) region from an infinite collection that is tileable by a given tileset,
we consider the following notion.
A finite region $\G$ is \emph{augmentable} if there is a finite s.c.\ region $\G'\supset\G$
such that both $\G'\bs\G$ and $\G'$ are tileable.

\problemdef{Augmentability}
{A tileset~$\T$ and a region $\G$.}
{Is $\G$ augmentable by~$\T$?}

As above, we can fix either $\T$ or $\G$ in the input (obviously not simultaneously).
We have the following result:

\begin{thm}
\problem{Augmentability} is undecidable.
Moreover, it remains undecidable even when $\G$ is fixed.
Similarly, there exists some tileset~$\T$ such that \problem{Augmentability} with $\T$ fixed is undecidable.
\end{thm}

\begin{proof}
We first prove undecidability in general, then briefly sketch how $\G$ or $\T$ may be fixed at the end of the proof.
Recall that it is undecidable whether a Turing machine with a $1$-way infinite tape halts.
By convention, the machine head starts at the left end of the tape, where the input is written.
Fix a Turing machine $\M$ on an alphabet $\Sigma$, with blank symbol $0\in\Sigma$, and starting state~$s$.
Let $\MM$ denote the associated tileset of \emph{machine tiles} afforded by Lemma~\ref{l:tm}.
Let $\B_\M$ be the tileset of \emph{border tiles} defined in \fig{aug}, consisting of
$\tf T$, $\tf{TR}$, $\L$, $\R$, $\tf{BL}$, and $\tf{BR}$,
$\L_x$, $\R_x$, and $\H_{x,q}$ for $x\in\Sigma$ and $q$ a halting state of~$\M$.
The tileset $\tsf F$ of \emph{filler tiles} consists of those in \fig{filler}.
Finally, replace unlabeled edges in $\T_\M=\MM\cup\B_\M\cup\tsf F$ with four new colors depending on which way each edge faces.
By abuse of language, we will continue to refer to these four colors on the boundary as white,
with the understanding that the white colors do not match,
and thus the border tiles can only be used on the boundary of a tiling and not in its interior.
This finishes the construction of the tileset~$\T=\T_\M$.

\newcommand{\inputtile}{
   \draw (-1,0) rectangle (5,1);
   \foreach \x in {0,1,...,4}
   {
      \draw (\x cm, 0 cm) -- (\x cm, 0.2 cm);
      \draw (\x cm, 1 cm) -- (\x cm, 0.8 cm);
   }
   \node at (2,0.5) {$\G_\x$};
}
\newcommand{\blanktile}{
   \draw (0,0) rectangle (1,2);
   \node at (0.5,0.5) {$\B$};
   \foreach \y in {1}
   {
      \draw (0 cm, \y cm) -- (0.2 cm, \y cm);
      \draw (1 cm, \y cm) -- (0.8 cm, \y cm);
   }
}
\newcommand{\stn}[3]{
   \stnn{#1}{#2}{\tf{#3}}
}
\newcommand{\stnn}[3]{
   \begin{scope}[xshift=#1 cm,yshift=#2 cm]
      \draw (0,0) rectangle (1,1);
      \node at (0.5,0.5) {$#3$};
   \end{scope}
}
\begin{figure}[hbtp]
   \begin{tikzpicture}
   \begin{scope}[xshift=-2cm,yshift=6cm]
   \inputtile

   \node[below] at (-0.5,0) {$L$};

   \node[below] at (0.5,0) {$(x_1,s)$};
   \foreach \x in {2,3,...,3}
      \node[below=4pt,xshift=\x cm] at (-0.5,0) {$x_\x$};
   \node[below=4pt] at (4.5,0) {$x_r$};


   \node[below=8pt] at (3.5,0) {$\ldots$};

   \node[right] at (5,0.5) {$T$};
   \end{scope}

   \tiletn{T}{}{0}{T}66{T}
   \tiletn{T}{}{R}{}96{TR}

   \tiletn{}{L}{L}{V}{-3}3{L}
   \tiletn{V}{R}{R}{}93{R}

   \tilen{L}{x}{}{L}00{\L_x}
   \tilen{L}{(x,q)}{}{R}30{\H_{x,q}}
   \tilen{R}{x}{}{R}60{\R_x}
   \tiletn{}{L}{}{L}{-3}0{BL}
   \tiletn{R}{R}{}{}90{BR}
   \end{tikzpicture}

   \caption{Border tiles $\B_\M$.}
   \label{fig:aug}
\end{figure}
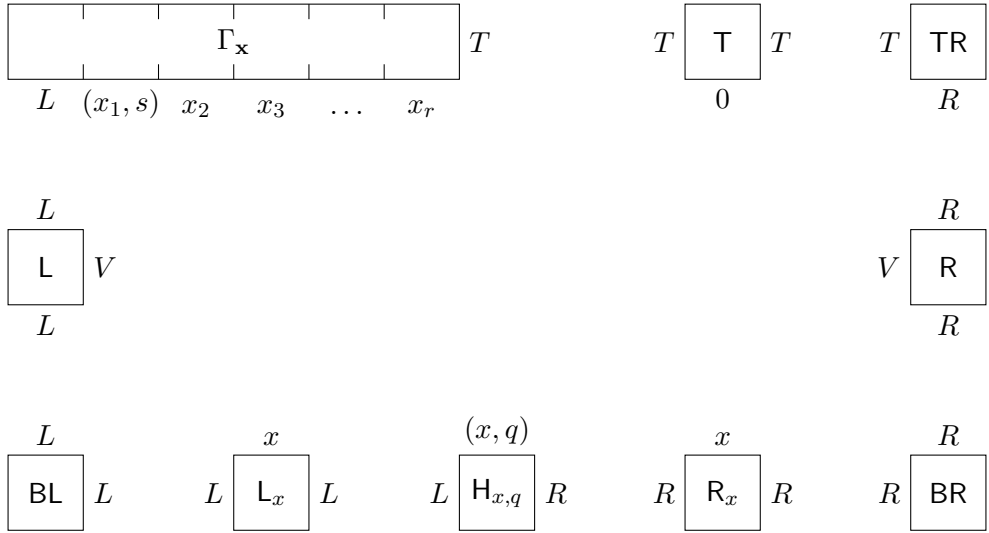

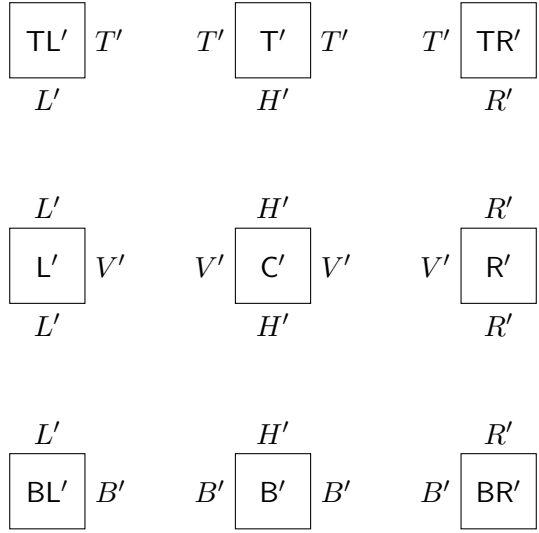
\begin{figure}[hbtp]
   \begin{tikzpicture}
   \tiletn{}{}{L'}{T'}06{TL'}
   \tiletn{T'}{}{H'}{T'}36{T'}
   \tiletn{T'}{}{R'}{}66{TR'}
   \tiletn{}{L'}{L'}{V'}03{L'}
   \tiletn{V'}{H'}{H'}{V'}33{C'}
   \tiletn{V'}{R'}{R'}{}63{R'}
   \tiletn{}{L'}{}{B'}00{BL'}
   \tiletn{B'}{H'}{}{B'}30{B'}
   \tiletn{B'}{R'}{}{}60{BR'}
   \end{tikzpicture}
   \caption{Filler tiles $\tsf F$.}
   \label{fig:filler}
\end{figure}

Let $\x=x_1x_2\ldots x_r\in\Sigma^*$ be a word in~$\Sigma$.
Construct region $\G=\G_\x$ as shown in \fig{aug}, where $s$ is the starting state of~$\M$,
and the white colors are subject to the same aforementioned replacement treatment.
It remains to show that $\G_\x$ is augmentable by $\T_\M$ if and only if $\M$ halts on input~$\x$.

\begin{figure}[hbtp]
   \begin{tikzpicture}
   \inputtile
   \foreach \x in {5,6,7}
   {
      \stn{\x}0{T}
   }
   \stn80{TR}

   \draw (0,0) rectangle (8,-4);
   \node at (4,-2) {space-time diagram of a halting computation};
   \node[above] at (3.5,-4) {$(x,q)$};

   \foreach \y in {-1,-2,...,-4}
   {
      \stn{-1}{\y}{L}
      \stn{8}{\y}{R}
   }

   \stnn{3}{-5}{\H_{x,q}}
   \foreach \x in {0,1,2}
   {
      \stnn{\x}{-5}{\L_\bullet}
   }
   \foreach \x in {4,5,6,7}
   {
      \stnn{\x}{-5}{\R_\bullet}
   }
   \stn{-1}{-5}{BL}
   \stn8{-5}{BR}

   \end{tikzpicture}
   \caption{Tiling of $\G'\bs\G$.}
   \label{fig:compute}
\end{figure}

First, suppose $M$ halts on input~$\x$.
Take a tiling representing the space-time diagram of a halting computation.
By definition, the top boundary color sequence of the tiling is~$\x$, possibly with blank symbol $0$ repeated on its right.
We may therefore place $\G$ above the space-time diagram, with tiles $\tf T$ to its right (see \fig{compute}).
Now we surround the region by the other border tiles to get a white rectangular region~$\G'$.
Of course, the bottom is tiled by~$\H_{x,q}$,
which matches the state symbol in the last row of the space-time diagram,
and with $\L_\bullet$ and $\R_\bullet$ tiles on the sides, where $\bullet$ denotes some unspecified symbol from~$\Sigma$.
We thus get a tiling of $\G'\bs\G$ automatically.
Notice that since $\G'$ is a white rectangle, it is tileable by the filler tiles $\tsf F\subset\T_\M$, as desired.

Conversely, suppose $\G_\x$ is augmentable by~$\T_\M$.
Since no tiles could be immediately above or to the left of~$\G$,
$\G'$ matches the upper left corner of~$\G$.
Thus in any tiling of $\G'$, the space currently occupied by $\G$ must be tiled by filler tiles.
Note that as the filler tiles can only be adjacent to other filler tiles,
and since $\G'$ is (simply) connected,
$\G'$ is tileable by filler tiles alone.
Of course, tilings of $\G'\bs\G$ must not use filler tiles, as they cannot be adjacent to~$\G$.
This means the boundary of $\G'$ must be colors common to both filler and non-filler tiles,
meaning that $\G'$ is a white rectangle.
The only way to tile $\G'\bs\G$ is by using the border tiles as shown in \fig{compute},
with precisely one $\H_{x,q}$ somewhere.
Since the interior cannot utilize border tiles, it must be tiled by the machine tiles~$\MM$ alone.
Thus it emulates a computation starting from $\x$, potentially followed by finitely many blank symbols,
and ending at a halting state (forced by the presence of a border tile $\H_{x,q}$ on the bottom boundary) after finitely many steps, as desired.

\medskip

We may interchange the role of $\G_\x$ and~$\tf{TL}'$,
adding $\G_\x$ to the tileset and removing $\tf{TL}'$ to be used as the input region.
This shows that \problem{Augmentability} remains undecidable even when $\G$ is fixed.
On the other hand, by fixing $\M$ as some Universal Turing machine,
we see that there is some fixed $\T$ such that \problem{Augmentability} is undecidable.
\end{proof}


\medskip

\section{Final remarks and open problems}\label{s:fin}

\subsection{} \label{ss:complexity}
Tileability of a given finite region is decidable simply by exhaustive search.
In this context,
tileability of finite regions by a finite tileset is \NP-complete~\cite{Lew}.%
\footnote{This was also indicated in an unpublished manuscript by Garey, Johnson, and Papadimitriou; see~\cite{GJ,Pap}.}
This is true even when the tileset is fixed as a horizontal bar $h_n$ of $n$ squares and a vertical bar $v_m$ of $m$ squares,
$n\geq2$ and $m\geq3$~\cite{BNRR}.%
\footnote{Note that everything is tileable if $n=1$; if $n=m=2$, it is a matching problem and thus is solvable in polynomial time (see \latin{e.g.}~\cite{LoP}).}
However, when the region is simply connected (s.c.),
tileability by $\{h_n,v_m\}$ can be determined in linear time~\cite{KK}.%
\footnote{Whenever the tileset is fixed, the complexity of the running time is given in terms of the area of the input region.}
By comparing these results, it is apparent that simple connectivity makes a difference.
This is in part due to Thurston's height function approach~\cite{Thu} and techniques in combinatorial group theory developed by Conway and Lagarias~\cite{CL}.
Using these methods, R\'emila showed that tileability of s.c.\ regions is quadratic for any two fixed rectangles~\cite{Rem}.
Perhaps surprisingly, there is a finite set of rectangular tiles such that tileability of s.c.~regions is \NP-complete~\cite{PY-rect}.
On the other hand, tileability of a rectangular region for a fixed set of tiles can be determined in linear time (see below).
Of course, if the rectangular region is unspecified, undecidability returns.
In a related direction, it is undecidable whether a tileset is a \emph{code}~\cite{BN-code},
\latin{i.e.}, every tileable finite region is uniquely tileable.

\subsection{} \label{ss:curious}
Theorem~\ref{t:rect} leads to some curious consequences.
\begin{thm}[\cite{DK}]
Given a tileset $\T$, there is a tileset $\T'$ consisting only of rectangular tiles such that $\T$ and $\T'$ tile the same rectangular regions.
\end{thm}
Note that $\T$ can tile \emph{some} rectangle if and only if $\T'$ is non-empty.
As such, in light of our main theorem, it is undecidable if $\T'$ is non-empty and,
\latin{a fortiori}, $\T'$ is not computable from~$\T$.
Indeed, while the proof in~\cite{DK} seems purely existential, now there is proof that it \emph{cannot} be made constructive.

\begin{thm}[\cite{LMP}]\label{t:lmp}
Tileability of an $[n\times m]$ rectangle by a fixed tileset of rectangles can be determined in time $O(\log n+\log m)$.
\end{thm}

Combining the two results, we see that there is a linear time algorithm for testing tileability of rectangles for any fixed tileset.
Again, the main theorem proves that
the linear time algorithm afforded by combining~\cite{DK} with~\cite{LMP} cannot be algorithmically determined if the tileset is not fixed in advance.
In the next subsection, we outline another consequence regarding the growth of certain functions.

\subsection{}
For a polyomino $\P$ that tiles some rectangle,
Klarner defined the \emph{order} $h(\P)$ of $\P$ as the minimum number of copies it takes to tile some rectangle~\cite{Kla},
and conjectured that there is no polyomino of order $3$, which has since been confirmed~\cite{SW}.
There are polyominoes of order $4s$ for each $s\in\N$~\cite{Gol-order},
but no non-rectangular polyomino whose order is odd%
\footnote{One should not confuse this with the notion of \emph{odd order}, also introduced in~\cite{Kla}.}
is known, and their existence is an open question (see \latin{e.g.}~\cite[\S15]{GO}).

For convenience, let us set $h(\P)=0$ if $\P$ does not tile any rectangle,
and let $$f(n)=\max\{h(\P):\abs{\P}=n\},$$ where $\abs{\P}$ is the area of (the shape of)~$\P$.
If it is computable or bounded above by a computable function,
then \problem{Rectangular Tileability} for a single polyomino would be decidable.
Indeed, if there exists a computable function $g(n)$ such that $f(n)\leq g(n)$,
then given $\P$, simply try tiling all rectangles with areas up to $g(\abs{\P})$.
This is a finite process that terminates.
If no tilings are found, then $\P$ does not tile any rectangle.
There are several claims that the problem is decidable (thus the function $f(n)$ is computable),
but to our best knowledge (see \latin{e.g.}~\cite{AS}), no proofs are currently known.

\newcommand{\h}{\widetilde h}
\newcommand{\f}{\widetilde f}
Along the same vein, let $\h(\T)$ denote the minimum area of a rectangle tileable by a tileset~$\T$,
and set $\h(\T)=0$ if $\T$ does not tile any rectangle.
Note that when $\T=\{\P\}$ is a single polyomino, we have $\h(\{\P\})=h(\P)/\abs{\P}$.
Similarly, define $$\f(n)=\max\{f(\T):\text{$\T$ has \emph{total} area $n$}\}.$$

\begin{cor}
The function $\f(n)$ grows faster than \emph{any} computable function, \latin{e.g.}, the fast-growing Ackermann function.
\end{cor}
\begin{proof}
This follows from the undecidability of \problem{Rectangular Tileability} (Theorem~\ref{t:rect}).
\end{proof}

\subsection{} \label{ss:wang}
Wang conjectured that there is no aperiodic tileset of Wang squares~\cite{Wang}.
If this conjecture were true, \problem{Tileability} for Wang squares would be decidable.
Indeed, simply enumerate \emph{all} tilings of $[N\times N]$ squares, $N=1,2,\ldots$.
Either some square admits no tiling (and thus the plane is not tileable),
or at some point a finite fundamental domain will be found that can be used to tile the whole plane periodically.
Berger disproved the conjecture by showing the undecidability of this tileability problem~\cite{Ber}.
As a consequence, there is an aperiodic tileset of Wang squares.
After a series of reductions in the number of Wang squares,
the current record is $13$~\cite{Cul}.

\subsection{}
Reducing the number of tiles in an aperiodic tileset has been the subject of much effort.
The famous Penrose tilings derive aperiodicity from a pair of tiles~\cite{Pen}.
Here a tile is a geometric shape that can be rotated (and reflected) in addition to translation.
These tilings inherit aperiodicity from quasicrystals and thus do not admit transformations onto lattices.
Socolar and Taylor constructed a single tile on the hexagonal lattice that exhibits aperiodic behavior when external matching rules are enforced~\cite{ST}.
Since the matching rules span non-adjacent tiles,
encoding this as a geometric tile (with no external matching rules)
necessarily resulted in a tile whose interior is disconnected.
There seems to be no obvious way to circumvent this in the plane.

Returning to the square lattice,
it is unknown whether there is a single polyomino that could force aperiodic behavior.
Decidability of whether a polyomino admits a tiling (or a periodic tiling) of the plane is open.
However, there is a polynomial-time algorithm for deciding whether a given polyomino admits an \emph{isohedral tiling}~\cite{KV},
\latin{i.e.}, a tiling where the group of isometries acts transitively on the tiles thereof.
If only translations are allowed on a single polyomino,
\problem{Tileability} of the plane is decidable~\cite{BN-single}, and an algorithm with running time quadratic in the boundary length is known~\cite{GV}.

\subsection{}
Let us consider tilings of finite regions.
Traditionally, the use of \emph{coloring maps} is a main tool for proving non-tileability (see~\cite{Gol-book}).
It actually proves the non-existence of \emph{signed tilings},
\latin{i.e.}, a finite collection of tiles with $\pm1$ weights
such that the weights of tiles covering each square sum to $1$ inside the region and $0$ outside (see~\cite{CL}).
Conversely, when there are no signed tilings, there is a coloring argument that certifies it (see~\cite{Pak}).
This is not true for ordinary tilings.
Indeed, there are non-tileable regions that admit signed tilings,
thus proving non-tileability using this method is impossible for those regions.

The notion of augmentability considered in Section~\ref{s:aug} can be formulated in this language of signed tilings,
by requiring tiles with equal weights to be pairwise disjoint.
Augmentability is therefore a (proper) intermediate notion between tileability and signed tileability.
That is, tileable regions are augmentable, and augmentable regions are signed tileable.
Moreover, the implications are not bidirectional~\cite{Korn}.

\subsection{}
There are instances where results of tilings in the plane do not extend to higher dimensions.
For example, the number of tilings of a finite region by dominoes can be counted efficiently in the plane,
but is \SP-complete in higher dimensions~\cite{PY-domino}.
However, all of the results in this paper extend easily to higher dimensions.
Indeed, \problem{Rectangular Tileability} extends by simply endowing each $2D$ tile with height $1$ in the remaining directions.
The results involving Turing machines can also be easily extended in the obvious manner.
Extending \problem{Complementary Tileability} with rectangles to higher dimensions is only slightly trickier;
the details will appear in~\cite{Yang}.

Another direction of generalization is to consider other infinite regions with reasonable finite descriptions.
For example, consider a periodic perforated plane,
where there are holes occurring throughout the region in some periodic manner.
This may lead to interesting results like Theorem~\ref{t:quadrant}.

One could also ask for the running-time complexity of the algorithm described in the proof of Theorem~\ref{t:alg}.
In particular, what happens if the tileset is the (horizontal and vertical) dominoes?

\vskip.7cm


\noindent
\textbf{Acknowledgements.} \,
I am grateful to my advisor Igor Pak for proposing these problems,
helpful conversations,
reading this paper, and providing invaluable feedback.
The work is supported by the NSF under Grant No.~DGE-0707424.

\bigskip
\bigskip

\bigskip
\bigskip
\appendix
\section{Geometric encoding of Wang tiles}\label{s:reduce}
Let $\cT$ be a collection of tiles and $\cR$ be a collection of regions.
A decision problem in \emph{$(\cT,\cR)$-\problem{Tileability}} consists
of a \emph{fixed} tileset $\T\subset\cT$,
receives some $\G\in\cR$ as input, and outputs whether $\G$ is $\T$-tileable.
We say $(\cT,\cR)$-\problem{Tileability} is \emph{linear time reducible} to $(\cT',\cR')$-\problem{Tileability}
if for any finite tileset $\T\subset\cT$, there exists a finite tileset $\T'\subset\cT'$
and a \emph{reduction map} $f:\cR\to\cR'$ that is computable in linear time (in the complexity of $\G\in\cR$),
such that $\G\in\cR$ is $\T$-tileable if and only if $f(\G)$ is $\T'$-tileable.
If, moreover, that $(\cT',\cR')$-\problem{Tileability} is linear time reducible to
$(\cT,\cR)$-\problem{Tileability}, then they are \emph{linear time equivalent}.
Note that the transformation of the tilesets need not be efficient nor bijective.
To simplify the notation, we drop the prefix in $(\cT,\cR)$-\problem{Tileability} when the sets~$\cT$ and~$\cR$ are understood.

\begin{lem}[Tileability Equivalence Lemma]\label{l:reduce}
For finite regions, the following classes of \problem{Tileability} problems are linear time equivalent:
\begin{enumerate}
\item \problem{Tileability} with a fixed set of rectangular tiles.
\item \problem{Tileability} with a fixed set of polyomino tiles.
\item \problem{Tileability} with a fixed set of generalized Wang tiles.
\item \problem{Tileability} with a fixed set of Wang squares.
\end{enumerate}
For cofinite regions, the last three are linear time equivalent.
\end{lem}

\begin{proof}
We may consider polyominoes as generalized Wang tiles with monochromatic boundary,
and break generalized Wang tiles into Wang squares by coloring each interior edge with unique colors not used elsewhere, as to force assembly.
The equivalence of tiling by rectangles with the other settings is proved in~\cite{PY-rect}.
It remains to show the reduction from Wang squares to polyominoes.

Given a set $\W$ of Wang squares.
There is an algorithm to construct a set $\T$ of polyominoes such that $\T$ tiles the plane if and only if $\W$ does.
Indeed, simply think of each tile in $\W$ as a huge square, and replace each edge by an appropriate rectilinear zig-zag to encode the edge color.
If done correctly, these boundary zig-zags will enforce the matching rule.

An explicit construction can be found in~\cite{Gol-reduce}.
We reproduce it here since we do need to investigate \emph{some} construction more closely.
The \emph{corner zig-zags} are made slightly more complicated than the original for clarity.

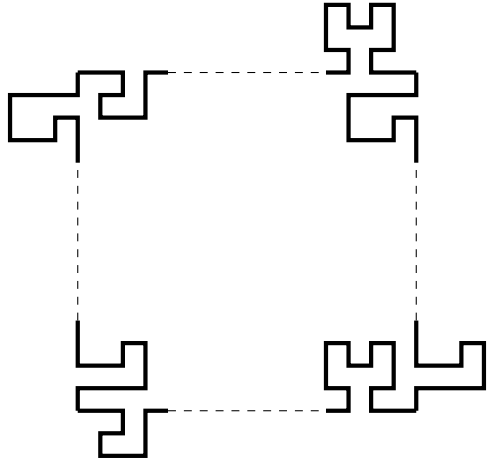
\begin{figure}[hbtp]
   \begin{tikzpicture}[scale=0.3]
   \newcommand{\mcsize}{15}
   \begin{scope}[ultra thick]
   \foreach \x in {0,\mcsize}
      \draw [xshift=\x cm,yshift=\mcsize cm] (0,0) -- (0,-1) -- (-3,-1) -- (-3,-3) -- (-1,-3) -- (-1,-2) -- (0,-2) -- (0,-4);
   \foreach \x in {0,\mcsize}
      \draw [xshift=\x cm] (0,0) -- (0,1) -- (3,1) -- (3,3) -- (2,3) -- (2,2) -- (0,2) -- (0,4);
   \foreach \y in {0,\mcsize}
      \draw [yshift=\y cm] (0,0) -- (2,0) -- (2,-1) -- (1,-1) -- (1,-2) -- (3,-2) -- (3,0) -- (4,0);
   \foreach \y in {0,\mcsize}
      \draw [xshift=\mcsize cm, yshift=\y cm] (0,0) -- (-2,0) -- (-2,1) -- (-1,1) -- (-1,3) -- (-2,3) -- (-2,2) -- (-3,2) -- (-3,3) -- (-4,3) -- (-4,1) -- (-3,1) -- (-3,0) -- (-4,0);
   \end{scope}

   \begin{scope}[dashed]
   \draw [yshift=-4cm] (0,8) -- (0,\mcsize);
   \draw [xshift=\mcsize cm,yshift=-4cm] (0,8) -- (0,\mcsize);
   \draw [xshift=-4cm] (8,0) -- (\mcsize,0);
   \draw [xshift=-4cm,yshift=\mcsize cm] (8,0) -- (\mcsize,0);
   \end{scope}
   \end{tikzpicture}
   \caption{Geometric encoding of a Wang square as a polyomino.}
   \label{fig:encode}
\end{figure}

Suppose there are $m$ edge colors used, and let $r>1+\log_2m$.
The dashed lines in \fig{encode} are of length~$r$.
Each edge color corresponds to a binary number with at most $r$ digits.
These digits are used to modify the dashed line segments.
A digit $0$ makes no modification,
while a digit $1$ adds a square outward in the corresponding position on the bottom and right,
and removes a square inward along the top and left.

The zig-zags on the corners force these tiles to align in a (dilated) square lattice grid.
Note that even if rotations and reflections are allowed,
the corner zig-zags moreover force all the tiles to have the same orientation.
Thus we \emph{could} in fact allow all isometries when tiling by polyominoes.
This fact is used in the proof of Corollary~\ref{cor:similar}.
When these polyominoes are adjacent, the series of one-square modifications on the touching boundaries enforce the color matching rules.
This construction clearly works for tiling finite or cofinite regions instead of the plane as well.
\end{proof}

\section{Turing machine emulation}\label{s:tm}
Consider a Turing machine $\M$ on an alphabet $\Sigma$ and $Q$ the set of states.
We refer to the elements of $\Sigma\sqcup\Sigma\times Q$ as \emph{colors},
where $(x,q)$, $x\in\Sigma$, $q\in Q$ are called \emph{compound colors}.
A configuration of $\M$ is represented as a (possibly infinite) sequence of colors,
where all but finitely many are the blank symbol $0\in\Sigma$.
Moreover, there must be precisely one compound color~$(x,q)$.
Reading $(x,q)$ as $x$, this sequence represents the contents of the tape.
The head of the machine is at the position of $(x,q)$, and the current state is~$q$.
Such a sequence of colors (corresponding to a valid configuration) is called \emph{valid}.
A row of (possibly infinite) Wang squares is \emph{valid} if their bottom colors form a valid sequence.
We say a set of Wang squares $\W$ \emph{emulates} $\M$ if given a valid row of Wang squares, and a second row of Wang squares below it,
the second row represents a configuration obtained from the previous one after one time step (and therefore is a valid row).
Moreover, if $\M$ is non-deterministic,
all possible next configurations are realizable by such tilings.
This way, a tiling with $\W$ is a space-time computation diagram of $\M$, given that it starts with a valid row.
Furthermore, all possible space-time diagrams are realizable as tilings.

\begin{lem}\label{l:tm}
There is a computable tileset $\MM$ of Wang squares that emulates a given Turing machine~$\M$.
\end{lem}

This is fairly standard; see \latin{e.g.}~\cite{LeP} for an explicit construction.

\end{document}